\documentclass[12pt,reqno]{amsart}
\usepackage{amsmath, amsthm, amssymb}

\advance \topmargin by -\headheight
\advance \topmargin by -\headsep
     
\evensidemargin \oddsidemargin
\newtheorem{theorem}{Theorem}[section]
\newtheorem{lemma}[theorem]{Lemma}
\newtheorem{corollary}[theorem]{Corollary}

\theoremstyle{definition}
\newtheorem{definition}[theorem]{Definition}

\theoremstyle{remark}

\numberwithin{equation}{section}

\newcommand{\mmod}[1]{\,\,(\text{mod}\,\,#1)}

\def\bfh{{\mathbf h}}

\def\bfn{{\mathbf n}}

\def\bfu{{\mathbf u}}

\def\bfw{{\mathbf w}}

\def\bfz{{\mathbf z}}

\def\calB{{\mathcal B}}

\def\calP{{\mathcal P}}

\def\calV{{\mathcal V}}

\def\Gtil{\widetilde G}

\def\dbN{{\mathbb N}}
\def\dbR{{\mathbb R}}\def\dbT{{\mathbb T}}
\def\dbZ{{\mathbb Z}}

\def\grm{{\mathfrak m}}\def\grM{{\mathfrak M}}
\def\grS{{\mathfrak S}}

\def\gry{{\mathfrak y}}

\def\alp{{\alpha}} \def\bfalp{{\boldsymbol \alpha}}
\def\bet{{\beta}}  
\def\gam{{\gamma}} \def\Gam{{\Gamma}}
\def\del{{\delta}} \def\Del{{\Delta}}
  
\def\bfeta{{\boldsymbol \eta}}
\def\tet{{\theta}}  

\def\lam{{\lambda}}  

 \def\bfXi{{\boldsymbol \Xi}}

\def\sig{{\sigma}}  

\def\Ups{{\Upsilon}} 
\def\bfPsi{{\boldsymbol \Psi}}
\def\bfPhi{{\boldsymbol \Phi}}
\def\ome{{\omega}} \def\Ome{{\Omega}}
\def\d{{\partial}}
\def\eps{\varepsilon}

\def\le{\leqslant} \def\ge{\geqslant}

\def\d{{\,{\rm d}}}

\begin{document}
\title[Weyl sums]{On Weyl sums for smaller exponents}
\author[Kent D. Boklan]{Kent D. Boklan}
\address{KDB: Department of Computer Science, Queen's College, Flushing, NY 11367, U.S.A.}
\email{boklan@boole.cs.qc.cuny.edu}
\author[Trevor D. Wooley]{Trevor D. Wooley$^*$}
\address{TDW: School of Mathematics, University of Bristol, University Walk, Clifton, 
Bristol BS8 1TW, United Kingdom}
\email{matdw@bristol.ac.uk}
\thanks{$^*$Supported by a Royal Society Wolfson Research Merit Award.}
\date{}
\begin{abstract} We present a hybrid approach to bounding exponential sums over $k$th powers via Vinogradov's mean value theorem, and derive estimates of utility for exponents $k$ of intermediate size.\end{abstract}
\maketitle

\section{Introduction} The main purpose of this paper is to present a new hybrid approach to bounding the modulus of the classical Weyl sum
$$f_k(\alp;P)=\sum_{1\le x\le P}e(\alp x^k),$$
where $e(z)$ denotes $e^{2\pi iz}$, for values of $\alp$ that are not well-approximated by rational numbers with a small denominator. Weyl \cite{Wey1916} was the first to successfully investigate bounds of this type in his seminal work concerning the uniform distribution of polynomial sequences. His methods, which involve the repeated squaring of the modulus of the exponential sum in combination with a consideration of the associated shift operator, still provide the sharpest estimates of their type for small values of $k$. Much stronger conclusions may be obtained for larger $k$ by bounding certain auxiliary mean values, as was shown by Vinogradov \cite{Vin1935}. Values of $k$ having intermediate size are of considerable interest in applications to Waring's problem and beyond, and our focus in this paper is on squeezing the very strongest bounds feasible from available estimates for these mean values.\par

In order to proceed further, we must introduce some notation. Write
$$g(\bfalp;P)=\sum_{1\le x\le P}e(\alp_1x+\alp_2x^2+\ldots +\alp_kx^k),$$
and define the mean value
$$J_{s,k}(P)=\int_{\dbT^k}|g(\bfalp;P)|^{2s}\d \bfalp ,$$
where $\dbT$ denotes the unit interval $[0,1)$. Estimates for $J_{s,k}(P)$ fall under the general appellation of {\it Vinogradov's mean value theorem}, and take the form
\begin{equation}\label{1.1}
J_{s,k}(P)\ll P^{2s-\frac{1}{2}k(k+1)+\Del_{s,k}},
\end{equation}
where $\Del_{s,k}$ is a real number depending on, at most, the positive integers $s$ and $k$. Aside from the latter quantities, in this paper implicit constants in Vinogradov's notation $\ll$ and $\gg$ will on occasion depend also on a positive number $\eps$. This convention we apply already in (\ref{1.1}). We say that an exponent $\Del_{s,k}$ is {\it permissible} when the estimate (\ref{1.1}) holds for all real numbers $P$. It may be shown that for all natural numbers $s$ and $k$ one has $\Del_{s,k}\ge 0$ (see \cite[equation (1.7)]{Bom1990}). A trivial estimate, meanwhile, demonstrates that there is no loss of generality in supposing that $\Del_{s,k}\le \frac{1}{2}k(k+1)$.\par

Next, let $k$ be a natural number, and consider a real parameter $\tet$ with $0\le \tet\le k/2$.  Let $\grm_\tet$ denote the set of real numbers $\alp$ having the property that, whenever $a\in \dbZ$ and $q\in \dbN$ satisfy $(a,q)=1$ and $|q\alp-a|\le P^{\tet-k}$, then one has $q>P^\tet$. In applications involving the Hardy-Littlewood (circle) method, one refers to $\grm_\tet$ as the set of {\it minor arcs} in the Hardy-Littlewood dissection. Constraints implicit in technology available for handling the complementary set of {\it major arcs} $\grM_\tet=[0,1)\setminus \grm_\tet$ dictate that the minor arcs $\grm_1$ are of special significance. Henceforth, we abbreviate $\grm_1$ to $\grm$, and $\grM_1$ to $\grM$. In \S2 we provide an estimate for $f_k(\alp;P)$ when $\alp$ belongs to the set of minor arcs $\grm$.

\begin{theorem}\label{theorem1.1} Let $k$ be a natural number with $k\ge 4$, and suppose that the exponent $\Del_{s,k-1}$ is permissible for $s\ge k$. Then for each $\eps>0$, one has
\begin{equation}\label{1.2}
\sup_{\alp\in \grm}|f_k(\alp;P)|\ll P^{1-\sig (k)+\eps},
\end{equation}
where
\begin{equation}\label{1.3}
\sig(k)=\max_{s\ge k}\left(\frac{3-\Del_{s,k-1}}{6s+2}\right).
\end{equation}
\end{theorem}

The familiar output of Vinogradov's method delivers a conclusion similar to that of Theorem \ref{theorem1.1}, but with the exponent $\sig (k)$ defined via the relation
$$\sig(k)=\max_{s\ge k}\left(\frac{1-\Del_{s,k-1}}{2s}\right).$$
Such a bound is immediate from \cite[Theorem 5.2]{Vau1997}, for example. The potential superiority of the conclusion of Theorem \ref{theorem1.1} may be discerned by noting that $\Del_{s,k-1}$ may now be permitted to be nearly three times as large, and still one obtains a minor arc estimate of the same strength as that available hitherto. Equipped with suitable estimates for the permissible exponents occurring in Vinogradov's mean value theorem, the formula (\ref{1.3}) may be converted into numerical values for the exponent $\sig(k)$. This we discuss in \S4, where we outline how to obtain the exponents listed in the following corollary.

\begin{corollary}\label{corollary1.2}
When $10\le k\le 20$, the estimate $(\ref{1.2})$ holds with $\sig(k)=\rho(k)^{-1}$, where $\rho(10)=440.87$, $\rho(11)=575.81$, $\rho(12)=733.58$, $\rho(13)=910.41$, $\rho(14)=1111.15$, $\rho(15)=1331.61$, $\rho(16)=1576.42$, $\rho(17)=1841.79$, $\rho(18)=2132.47$, $\rho(19)=2444.02$, $\rho(20)=2781.54$.
\end{corollary}

By way of comparison, Parsell \cite{Par2009}, improving slightly on Ford \cite{For1995}, has obtained a similar conclusion with $\rho(11)=743.409$, $\rho(12)=999.270$, $\rho(13)=1223.475$, $\rho(14)=1420.574$, $\rho(15)=1632.247$, $\rho(16)=1856.535$, $\rho(17)=2114.819$, $\rho(18)=2436.255$, $\rho(19)=2779.680$, $\rho(20)=3150.605$. Our conclusions are inferior to those stemming from Weyl's inequality  for $k\le 9$, for the latter shows that (\ref{1.2}) holds with $\sig(k)^{-1}=2^{k-1}$ (see \cite[Lemma 2.4]{Vau1997}). Indeed, our methods provide the exponent $\sig(9)=\rho(9)^{-1}$ with $\rho(9)=324.00$, whereas Weyl's inequality yields $\rho(9)=256$. On the other hand, while the exponents obtained by Parsell, and by Ford, are inferior to the Weyl exponent $\rho(10)=512$, our exponent $\rho(10)=440.87$ is superior. We should remark also that the conclusion of Theorem \ref{theorem1.1} has no impact on the sharpest asymptotic bound at the time of writing, namely $\sig(k)^{-1}=(\frac{3}{2}+o(1))k^2\log k$ (see \cite{Woo1995}).\par

When $k\ge 6$ and $\alp\in \grm_3$, work of Heath-Brown \cite{HB1988} supplies a bound of the shape $|f_k(\alp;P)|\ll P^{1-\tau(k)+\eps}$, with $\tau(k)^{-1}=3\cdot 2^{k-3}$. At present, a successful analysis of $f_k(\alp;P)$ for $\alp$ in the complementary set $\grM_3$ is in general beyond our competence, and so although our exponent $\rho(10)=440.87$ is inferior to the exponent $\tau(10)^{-1}=384$ associated with Heath-Brown's estimate, the latter is limited in its application. We refer the reader to \cite{Bok1994} for more on this matter.\par

We briefly here illustrate some consequences of Corollary \ref{corollary1.2} by considering the expected asymptotic formula in Waring's problem. Define $R_{s,k}(n)$ to be the number of representations of the natural number $n$ as the sum of $s$ $k$th powers of positive integers. Also, denote by $\grS_{s,k}(n)$ the associated {\it singular series}
$$\grS_{s,k}(n)=\sum_{q=1}^\infty \sum^q_{\substack{a=1\\ (a,q)=1}}\Bigl( q^{-1}\sum_{r=1}^qe(ar^k/q)\Bigr)^se(-na/q).$$
We define $\Gtil(k)$ to be the least integer $s_0$ for which, whenever $s\ge s_0$, one has
\begin{equation}\label{1.6}
R_{s,k}(n)=\frac{\Gam(1+1/k)^s}{\Gam(s/k)}\grS_{s,k}(n)n^{s/k-1}+o(n^{s/k-1}).
\end{equation}
Subject to modest congruence conditions, one has $1\ll \grS (n)\ll n^\eps$, and so the relation (\ref{1.6}) does indeed constitute an honest asymptotic formula (see \cite[Chapter 4]{Vau1997}). In \S4 we indicate how to establish the following bounds.

\begin{corollary}\label{corollary1.3}
One has $\Gtil(9)\le 365$, $\Gtil(10)\le 497$, $\Gtil(11)\le 627$, $\Gtil(12)\le 771$, $\Gtil(13)\le 934$, $\Gtil(14)\le 1112$, $\Gtil(15)\le 1307$, $\Gtil(16)\le 1517$, $\Gtil(17)\le 1747$, $\Gtil(18)\le 1992$, $\Gtil(19)\le 2255$, $\Gtil(20)\le 2534$.
\end{corollary}

For comparison, the sharpest bounds available hitherto are $\Gtil(9)\le 393$, $\Gtil(10)\le 551$, due to Ford \cite{For1995}, and $\Gtil(11)\le 706$, $\Gtil(12)\le 873$, $\Gtil(13)\le 1049$, $\Gtil(14)\le 1231$, $\Gtil(15)\le 1431$, $\Gtil(16)\le 1645$, $\Gtil(17)\le 1879$, $\Gtil(18)\le 2134$, $\Gtil(19)\le 2410$, $\Gtil(20)\le 2701$, due to Parsell \cite{Par2009}. Our methods establish that $\Gtil(8)\le 233$, which is inferior to the first author's bound $\Gtil (8)\le 224$ (see \cite{Bok1994}). We would be remiss to not also mention the bounds $\Gtil(k)\le 2^k$ $(k\ge 3)$ due to Vaughan \cite{Vau1986a, Vau1986b}, and $\Gtil(k)\le \frac{7}{8}2^k$ $(k\ge 6)$ due to the first author \cite{Bok1994}. The asymptotic situation remains unchanged at the time of writing, with Ford's bound $\Gtil(k)\le k^2(\log k+\log \log k+O(1))$ valid for large $k$ (see \cite{For1995}).\par

Our principal conclusion, the minor arc estimate in Theorem \ref{theorem1.1}, is obtained by applying a variant of the Bombieri-Korobov estimate in combination with a major arc estimate due to Vaughan. In essence, the former estimate provides an estimate for $\sup_{\alp \in \grm_2}|f_k(\alp;P)|$, whilst the latter permits us to prune the set $\grM_2$ back to $\grM_1$, so that we are left with an upper bound for $\sup_{\alp \in \grm_1}|f_k(\alp;P)|$. The details will be found in \S2.\par

Some words are in order concerning the calculation of permissible exponents $\Del_{s,k}$. Forthcoming work of the second author transforms the landscape so far as bounds for the mean value $J_{s,k}(P)$ are concerned, and so it seems an unwarranted indulgence to invest too much space in explaining the nuances of various refinements in the underlying iterative method used in this paper. We have therefore chosen to focus on the ideas underpinning Theorem \ref{theorem1.1}, and to sketch two refinements to the iterative method in outline so that such ideas are not lost to the literature. Thus, in \S3, the reader will find a sketch of the changes necessary to replace the classical iteration which bounds $J_{s+k,k}(P)$ in terms of $J_{s,k}(P)$, by one which just as efficiently bounds $J_{s+k-1,k}(P)$ in terms of $J_{s,k}(P)$. Likewise, a modest refinement that with successive efficient differences reduces the number of variables differenced, so as to more efficiently make use of underlying congruences, is also outlined. Detailed treatment of these refinements would be otiose.\par

Throughout this paper, the letter $k$ will denote an arbitrary integer exceeding $1$, the letter $s$ will denote a positive integer, and $\eps$ will denote a sufficiently small positive number. We take $P$ to be a large real number depending at most on $k$, $s$ and $\eps$, unless otherwise indicated. In an effort to simplify our analysis, we adopt the following convention concerning the number $\eps$. Whenever $\eps$ appears in a statement, either implicitly or explicitly, we assert that the statement holds for each $\eps>0$. Note that the ``value'' of $\eps$ may consequently change from statement to statement. 

\section{Estimates of Weyl type} Our proof of Theorem \ref{theorem1.1} makes use of a special case of an estimate of Bombieri (see \cite[Theorem 8]{Bom1990}) that improves on earlier work of Korobov \cite{Kor1958}. In order to describe this result, we introduce some additional notation. When $b$ and $r$ are natural numbers, and $\bfn\in \dbZ^r$, denote by $\Ups_{b,r}(\bfn;P)$ the number of integral solutions of the system of equations
$$\sum_{i=1}^bm_i^j=n_j\quad (1\le j\le r),$$
with $1\le m_i\le P$ $(1\le i\le b)$, and then put
$$\Ups_{b,r}(P)=\max_{\bfn\in \dbZ^r}\Ups_{b,r}(\bfn;P).$$
In addition, write
\begin{equation}\label{2.1}
\Ome_r(q,P)=\prod_{j=1}^r(P^{-j}+P^{j-k}+q^{-1}+qP^{-k}).
\end{equation}

\begin{lemma}\label{lemma2.1} Let $b$, $k$ and $r$ be natural numbers with $1\le r\le k-1$. In addition, suppose that $\alp$ is a real number, and that $a\in \dbZ$ and $q\in \dbN$ satisfy $(a,q)=1$ and $|\alp-a/q|\le q^{-2}$. Then one has
$$f_k(\alp;P)\ll P\left( P^{kr-b}\Ups_{b,r}(P)\Ome_r(q,P)J_{s,k-1}(P)/J_{s,k-r-1}(P)\right)^{1/2bs}.$$
\end{lemma}

\begin{proof} This is immediate from \cite[Theorem 8]{Bom1990}.\end{proof}

The interested reader may care to compare Lemma \ref{lemma2.1} with Theorem 1.1 of \cite{Par2009}, the latter potentially having greater flexibility. We apply Lemma \ref{lemma2.1} with $r=2$ and $b=3$ in order to bound $|f_k(\alp;P)|$ for $\alp\in \grm_\tet$ when $1\le \tet\le 2$.

\begin{lemma}\label{lemma2.2}
Let $\del$ be a real number with $0\le \del\le 1$. In addition, let $s$ and $k$ be natural numbers with $s\ge k\ge 4$, and suppose that the exponent $\Del_{s,k-1}$ is permissible. Then one has
$$\sup_{\alp\in \grm_{2-\del}}|f_k(\alp;P)|\ll P^{1-\nu (s,k)+\eps},$$
where
\begin{equation}\label{2.1b}
\nu(s,k)=\frac{3-\del-\Del_{s,k-1}}{6s}.
\end{equation}
\end{lemma}

\begin{proof} Suppose that $k\ge 4$ and $\alp \in \grm_{2-\del}$. Then as a consequence of Dirichlet's theorem on Diophantine approximation, there exist $a\in \dbZ$ and $q\in \dbN$ with $(a,q)=1$, $1\le q\le P^{k-2+\del}$ and $|q\alp -a|\le P^{2-\del -k}\le q^{-1}$. The definition of $\grm_{2-\del}$ ensures that $q>P^{2-\del}$, and thus it follows from (\ref{2.1}) that
$$\Ome_2(q,P)\ll (P^{-1}+P^{\del-2})(P^{-2}+P^{\del-2})\ll P^{\del -3}.$$
Suppose that $\Del_{s,k-1}$ is a permissible exponent, so that
$$J_{s,k-1}(P)\ll P^{2s-\frac{1}{2}k(k-1)+\Del_{s,k-1}}.$$
Then in view of the lower bound $J_{s,k-3}(P)\gg P^{2s-\frac{1}{2}(k-2)(k-3)}$, which follows from the non-negativity of permissible exponents $\Del_{s,k-3}$, we deduce from Lemma \ref{lemma2.1} that
\begin{align}
f_k(\alp;P)&\ll P\left( P^{2k-3}\Ups_{3,2}(P)\Ome_2(q,P)P^{\frac{1}{2}(k-2)(k-3)-\frac{1}{2}k(k-1)+\Del_{s,k-1}}\right)^{1/6s}\notag \\
&\ll P\left( P^{\Del_{s,k-1}+\del -3}\Ups_{3,2}(P)\right)^{1/6s}.\label{2.2}
\end{align}

\par We next bound the quantity $\Ups_{3,2}(P)$. Let $n_1$ and $n_2$ be integers, and consider the number of integral solutions of the simultaneous equations
\begin{align}
m_1^2+m_2^2+m_3^2&=n_2,\label{2.3a}\\
m_1+m_2+m_3&=n_1,\label{2.3b}
\end{align}
with $1\le m_i\le P$ $(1\le i\le 3)$. Eliminating the variable $m_3$ between (\ref{2.3a}) and (\ref{2.3b}), we deduce that $3X^2+Y^2=N$, where we have written
\begin{equation}\label{2.5}
X=2m_1+m_2-n_1,\quad Y=3m_2-n_1\quad \text{and}\quad N=6n_2-2n_1^2.
\end{equation}
But the number of integer solutions $X,Y$ of this equation is $O((|N|+1)^\eps)$ (see, for example, Estermann \cite{Est1931}). For each fixed choice of $X,Y$, the equations (\ref{2.5}) may be solved uniquely for $m_1$ and $m_2$, and then the value of $m_3$ is determined uniquely by the linear equation (\ref{2.3b}). Thus we deduce that
$$\Ups_{3,2}(\bfn;P)\ll (|n_1|+|n_2|+1)^\eps.$$
However, the simultaneous equations (\ref{2.3a}), (\ref{2.3b}) plainly possess no solutions when $|n_2|>3P^2$, or when $|n_1|>3P$, and thus we conclude that
\begin{equation}\label{2.6}
\Ups_{3,2}(P)\ll \max_{|n_1|\le 3P}\max_{|n_2|\le 3P^2}(|n_1|+|n_2|+1)^\eps \ll P^{3\eps}.
\end{equation}

\par Substituting (\ref{2.6}) into (\ref{2.2}), we deduce that $f_k(\alp;P)\ll P^{1-\nu (s,k)+\eps}$, where $\nu(s,k)$ is defined as in (\ref{2.1b}), and the conclusion of the lemma follows.
\end{proof}

We next apply major arc estimates to prune the set $\grM_2$ down to $\grM_1$.

\begin{lemma}\label{lemma2.3}
Let $\del$ be a real number with $0\le \del\le 1$. Then for any natural number $k$ with $k\ge 3$, one has
$$\sup_{\alp\in \grm_1\setminus \grm_{2-\del}}|f_k(\alp;P)|\ll P^{1-1/k}+P^{1-\del/2+\eps}.$$
\end{lemma}

\begin{proof} When $a\in \dbZ$, $q\in \dbN$ and $\bet\in \dbR$, define
$$S(q,a)=\sum_{r=1}^qe(ar^k/q)\quad \text{and}\quad v(\bet)=\int_0^P e(\bet \gam^k)\d \gam .$$
Then from \cite[Theorem 4.1]{Vau1997}, one finds that when $\alp \in \dbR$, $a\in \dbZ$, $q\in \dbN$ and $(a,q)=1$, one has
\begin{equation}\label{2.7}
f_k(\alp;P)-q^{-1}S(q,a)v(\alp-a/q)\ll q^\eps (q+P^k|q\alp-a|)^{1/2}.
\end{equation}
Moreover, from \cite[Theorems 4.2 and 7.3]{Vau1997}, one sees that
\begin{equation}\label{2.8}
q^{-1}S(q,a)v(\alp-a/q)\ll P(q+P^k|q\alp-a|)^{-1/k}.
\end{equation}

\par Consider a real number $\alp \in \grm_1\setminus \grm_{2-\del}$. By Dirichlet's approximation theorem together with the hypothesis that $\alp \not\in \grm_{2-\del}$, there must exist $a\in \dbZ$ and $q\in \dbN$ with $(a,q)=1$ and $|q\alp-a|\le P^{2-\del -k}$ for which $q\le P^{2-\del}$. But $\alp \in \grm_1$, and so one has either $|q\alp-a|>P^{1-k}$ or $q>P$. One therefore finds that
$$P<q+P^k|q\alp-a|\le 2P^{2-\del}.$$
Consequently, in view of (\ref{2.7}) and (\ref{2.8}), one obtains
$$f_k(\alp;P)\ll P^{1-1/k}+P^\eps(P^{2-\del})^{1/2},$$
and the conclusion of the lemma is immediate.
\end{proof}

\begin{proof}[The proof of Theorem \ref{theorem1.1}] Let $s$ and $k$ be natural numbers with $s\ge k\ge 4$, and suppose that the exponent $\Del_{s,k-1}$ is permissible. We define $\del=\del(s,k)$ by
$$\del(s,k)=\frac{3-\Del_{s,k-1}}{3s+1}.$$
The hypothesis $s\ge k$ ensures that $\del\le 1/k$. We claim that
$$\sup_{\alp\in \grm}|f_k(\alp;P)|\ll P^{1-\del/2+\eps}.$$
When $\del<0$, this assertion follows from the trivial estimate $|f_k(\alp;P)|\le P$. We may therefore suppose that $0<\del\le 1/k$. In such circumstances, it follows from Lemma \ref{lemma2.2} that
$$\sup_{\alp\in \grm_{2-\del}}|f_k(\alp;P)|\ll P^{1-\nu(s,k)+\eps},$$
where
$$\nu(s,k)=\frac{3-\Del_{s,k-1}}{3s+1}\left( \frac{3s+1}{6s}-\frac{1}{6s}\right) =\frac{3-\Del_{s,k-1}}{6s+2}=\del/2.$$
On the other hand, from Lemma \ref{lemma2.3} one finds that
$$\sup_{\alp \in \grm_1\setminus \grm_{2-\del}}|f_k(\alp;P)|\ll P^{1-1/k}+P^{1-\del/2+\eps}.$$
Since $\grm_1=\grm_{2-\del}\cup (\grm_1\setminus \grm_{2-\del})$, we infer that
$$\sup_{\alp\in \grm_1}|f_k(\alp;P)|\ll P^{1-\del/2+\eps}.$$
This confirms our earlier assertion, and from here the conclusion of Theorem \ref{theorem1.1} follows on noting that $2\sig(k)=\underset{s\ge k}\max\,\del(s,k)$.
\end{proof}

\section{Improvements in Vinogradov's mean value theorem} The primary objective of this section is to sketch certain modest improvements to the efficient differencing method in Vinogradov's mean value theorem. These developments deliver the following conclusion.

\begin{theorem}\label{theorem3.1} Let $t$ and $k$ be natural numbers with $t\ge k\ge 2$, and suppose that the exponent $\mu $ satisfies $2t-\frac{1}{2}k(k+1)<\mu \le 2t$ and $J_{t,k}(P)\ll_{t,k}P^\mu $. When $s=t+l(k-1)$ $(l\in \dbN )$, define $\lam_s$, $\Del_s$, $\tet_s$ and $\phi (j,s,J)$ recursively as follows. Put $\Del_t=\mu +\frac{1}{2} k(k+1)-2t$. Then, for $j=1,\dots ,k$, put $\phi (j,s,j)=1/k$, and evaluate $\phi (j,s,J-1)$ successively for $J=j,\dots 2$ by putting
\begin{equation}\label{3.1}
\phi^*(j,s,J-1)=\frac{1}{2k}+\left( \frac{1}{2}+\frac{{\textstyle{\frac{1}{2}}}(J-1)(J-2)-\Del_s}{2k(k-J+1)}\right)\phi (j,s,J),
\end{equation}
and
$$\phi (j,s,J-1)=\min \left\{ 1/k,\phi^*(j,s,J-1)\right\} .$$
Finally, set
$$\tet_s=\min_{1\le j\le k}\phi (j,s,1),$$
$$\Del_s=\Del_{s-k+1}(1-\tet_s)+(k-1)(k\tet_s-1),$$
$$\lam_s=2s-{\textstyle{\frac{1}{2}}}k(k+1)+\Del_s.$$
Then for each natural number $s=t+l(k-1)$ $(l\in \dbN )$, one has $J_{s,k}(P)\ll P^{\lam_s}$. 
\end{theorem}

We note that a similar conclusion was obtained in \cite[Theorem 1.1]{Woo1992}, save with $s=t+lk$ in place of $s=t+l(k-1)$, and with the denominator $2k(k-J+1)$ in (\ref{3.1}) replaced by $2k^2$. Note, in particular, that the first of these adjustments enhances the efficiency of the method by a scale factor of roughly $(1-1/k)^{-1}$. The second adjustment also represents an improvement, because in applications one makes a choice of $j$ for which $\frac{1}{2}(j-1)(j-2)<\Del_s$.\par

As we have stressed, forthcoming work of the second author makes it desirable to provide the minimum of detail in our discussion here. We refer the reader to \cite{Woo1992} for a discussion of preliminaries and any unexplained notation. We begin here by recalling a definition.

\begin{definition} Let $d$ and $k$ be integers with $0\le d\le k$. Let $P$ be a positive real parameter, and let $A$ be a sufficiently large (but fixed) positive real number. Then we say that the $k$-tuple of polynomials $(\bfPsi )=(\Psi_1(x),\ldots ,\Psi_k(x))\in \dbZ [x]^k$ is of type $(d,P,A)$ if\vskip.0cm
(a) $\Psi_i$ has degree $i-d$ for $i\ge d$, and is identically zero for $i<d$, and\vskip.0cm
(b) the coefficient of $x^{i-d}$ in $\Psi_i(x)$ is non-zero, and bounded above by $AP^d$ $(1\le i\le k)$.
\end{definition}

When the system $(\bfPsi )$ is of type $(d,P,A)$, we write
$$J(\bfPsi ;\bfz )=\text{det}\left( \frac{\partial \Psi_{i+d}(z_j)}{\partial z_j}\right)_{1\le i,j\le k-d},$$
and denote by $\calB (p;\bfu ;\bfPsi )$ the number of solutions of the system of congruences
\begin{equation}\label{3.3}
\sum_{i=1}^{k-d}\Psi_j(z_i)\equiv u_j\pmod{p^j}\quad (d+1\le j\le k),
\end{equation}
with $1\le z_i\le p^k$ $(1\le i\le k-d)$ and $(J(\bfPsi;\bfz),p)=1$. Also, we define $\ome (k,d)={\textstyle{\frac{1}{2}}}(k-d)(k-d-1)$.

\begin{lemma}\label{lemma3.2} Suppose that the system $(\bfPsi)$ is of type $(d,P,A)$. Then one has $\calB (p;\bfu ;\bfPsi )\ll p^{\ome (k,d)}$, where the implicit constant depends only on $k$.
\end{lemma}

\begin{proof} We apply the same argument as in the proof of \cite[Lemma 2.2]{Woo1992} with the singular exception that, since the congruences (\ref{3.3}) have only $k-d$ variables in place of $k$, the factor $p^{kd}$ in \cite[equation (2.4)]{Woo1992} may be deleted.
\end{proof}

As usual, we take $P$ to be our basic parameter, a sufficiently large positive real number. Suppose that $(\bfPsi)$ is of type $(d,P,A)$. We define the integer $d^*$ associated to $d$ by
$$d^*=\begin{cases} d,&\text{when $d\ge 1$},\\ 1,&\text{when $d=0$.}\end{cases}$$
Consider the quantity
$$\Ome=\sup_\bfz \left( \frac{\log |J(\bfPsi ;\bfz )|}{\log P}\right),$$
where the supremum is over $\bfz $ with $1\le z_i\le P$ $(1\le i\le k-d)$ and $J(\bfPsi ;\bfz )\ne 0$. Plainly, there exists a positive integer $l=l(A,k)$, independent of $P$, such that $\Ome<k^l$. Then, with $\tet $ a real number with $0<\tet \le 1/k$, we take $\calP (\tet )$ to be the set consisting of the smallest $[2k^l/\tet ]+1$ prime numbers exceeding $P^\tet $. Upon taking $P$ sufficiently large, we have $P^\tet <p<2P^\tet $ for each $p\in \calP (\tet )$.\par

When $0\le d\le k$, denote by $K_{s,d}(P,Q;\bfPsi )$ the number of integral solutions of the system
$$\sum_{n=1}^{k-d^*}\left( \Psi_i(z_n)-\Psi_i(w_n)\right) +\sum_{m=1}^s(x_m^i-y_m^i)=0\quad (1\le i\le k),$$
with
\begin{equation}\label{3.5}
1\le z_n,w_n\le P\quad (1\le n\le k-d^*),\quad 1\le x_m,y_m\le Q\quad (1\le m\le s).
\end{equation}
Also, when $p\in \calP (\tet )$, define $L_{s,d}(P,Q;\tet ;p;\bfPsi )$ to be the number of integral solutions of the system
$$\sum_{n=1}^{k-d^*}\left( \Psi_i(z_n)-\Psi_i(w_n)\right) +p^i\sum_{m=1}^s\left( u_m^i-v_m^i\right) =0\quad (1\le i\le k),$$
with $\bfz ,\bfw$ satisfying (\ref{3.5}), and
$$0<u_m,v_m\le QP^{-\tet }\quad (1\le m\le s),\quad z_n\equiv w_n\mmod{p^k}\quad (1\le n\le k).$$
We then put
$$L_{s,d}(P,Q;\tet ;\bfPsi )=\max_{p\in \calP (\tet )}L_{s,d}(P,Q;\tet ;p;\bfPsi ).$$
We are, at last, prepared to state the fundamental lemma.

\begin{lemma}\label{lemma3.3} Suppose that $s\ge d\ge 1$, $P^\tet \le Q\le P$, and that $(\bfPsi)$ is a system of type $(d,P,A)$. Then there exists a system $(\bfPhi)$ of the same type for which
$$K_{s,d}(P,Q;\bfPsi )\ll_{\tet ,A}P^{k-d^*}J_{s,k}(Q)+P^{(2s+\ome (k,d^*)-d^*)\tet }L_{s,d}(P,Q;\tet ;\bfPhi ).$$
\end{lemma}

\begin{proof} The argument of the proof of \cite[Lemma 3.1]{Woo1992} may be applied in the present context, the modified definitions of $K_s$ and $L_s$ generating only superficial differences.
\end{proof}

We add to this lemma an initial procedure to initiate the iteration. 

\begin{lemma}\label{lemma3.a}
There exists a system $(\bfPhi)$ of type $(0,P,1)$ such that
$$J_{s+k-1,k}(P)\ll P^{k-1}J_{s,k}(P)+P^{(2s+\ome(k,1)-1)\tet}L_{s,0}(P,P;\tet;\bfPhi).$$
\end{lemma}

\begin{proof} The argument leading to \cite[equation (3.15)]{Woo1992} ensures that
\begin{equation}\label{3.8}
J_{s+k-1,k}(P)\ll T_1+p^{2s-2}\max_{1\le x\le p}T_2(x),
\end{equation}
where
$$T_1=\int_{\dbT^k}|f_k(2\bfalp;P)^2f_k(\bfalp;P)^{2s+2k-6}|\d\bfalp ,$$
and $T_2(x)$ denotes the number of solutions of the system of equations
\begin{equation}\label{3.9}
\sum_{n=1}^k(z_n^i-w_n^i)+p^i\sum_{m=1}^{s-1}(u_m^i-v_m^i)=0\quad (1\le i\le k),
\end{equation}
with $-x/p<u_m,v_m\le (P-x)/p$ $(1\le m\le s-1)$, and $1\le z_n,w_n\le P$ $(1\le n\le k)$ subject to $(J(\bfPsi;\bfz),p)=(J(\bfPsi;\bfw),p)=1$. The reader should inspect part (i) of the proof of \cite[Lemma 3.1]{Woo1992}, together with \cite[equation (3.10)]{Woo1992}, for the necessary ideas, and should note that in the present context we take $\Psi_i(z)=z^i$ $(1\le i\le k)$. Thus the system $(\bfPsi)$ is of type $(0,P,1)$.\par

In view of the non-singularity hypothesis imposed on $\bfz$ and $\bfw$, the system of congruences
$$\sum_{n=1}^kz_n^i\equiv \sum_{n=1}^kw_n^i\pmod{p}\quad (1\le i\le k),$$
implicit in (\ref{3.9}), imply that the sets $\{z_1,\ldots ,z_k\}$ and $\{ w_1,\ldots ,w_k\}$ are equal modulo $p$. There is no loss of generality in supposing then that $z_n\equiv w_n\pmod{p}$ $(1\le n\le k)$, provided that we inflate our estimates by the combinatorial factor $k!$, which is harmless. The non-singularity hypothesis ensures, moreover, that $z_1,\ldots ,z_k$ are distinct modulo $p$, and likewise $w_1,\ldots ,w_k$. The solutions are now of two types. There are the solutions counted by $T_2(x)$ in which $p|z_n$ for some index $n$, and those in which $p|z_n$ for no index $n$. In the former case, we relabel variables so that $n=k$, and then define $u_s$ and $v_s$ by putting $pu_s=z_k$ and $pv_s=w_k$. In the latter case, the number of solutions may be estimated by applying H\"older's inequality to an associated mean value of exponential sums. The strategy here is similar to that which leads to \cite[equation (3.6)]{Woo1992}. We have restricted $2s-2$ of the variables to the congruence class zero modulo $p$, and we have a further congruence class $\xi$ modulo $p$ for $z_k$ and $w_k$. By applying H\"older's inequality, we are able to force all of these variables to lie in the same congruence class modulo $p$, at the cost of an additional factor $p$ in our estimates. In this way, one finds that
\begin{equation}\label{3.9a}
T_2(x)\ll p\max_{1\le \xi\le p}T_3(\xi),
\end{equation}
wherein $T_3(\xi)$ denotes the number of integral solutions of the system
\begin{equation}\label{3.10}
\sum_{n=1}^{k-1}(z_n^i-w_n^i)+\sum_{m=1}^s((pu_m+\xi)^i-(pv_m+\xi)^i)=0\quad (1\le i\le k),
\end{equation}
with $-\xi/p<u_m,v_m\le (P-\xi)/p$ $(1\le m\le s)$, and $1\le z_n,w_n\le P$ $(1\le n\le k-1)$ subject to the additional condition that, with $\gry$ equal either to $z$ or $w$, one has (i) $\gry_n\not\equiv \xi\pmod{p}$ for $1\le n\le k-1$, and (ii) $\gry_u\equiv \gry_v\pmod{p}$ for no $u$ and $v$ with $1\le u<v\le k-1$.\par

By the Binomial Theorem, the system (\ref{3.10}) is equivalent to
$$\sum_{n=1}^{k-1}((z_n-\xi)^i-(w_n-\xi)^i)=p^i\sum_{m=1}^s(u_m^i-v_m^i)\quad (1\le i\le k).$$
For a fixed $(k-1)$-tuple $\bfh$, the number of solutions of the system of congruences
$$\sum_{n=1}^{k-1}(z_n-\xi)^i\equiv h_i\pmod{p^i}\quad (2\le i\le k),$$
with $\bfz$ satisfying the non-singularity conditions (i) and (ii) above, and $1\le z_n\le p^k$ $(1\le n\le k-1)$, is readily confirmed to be at most $(k-1)!p^{\frac{1}{2}(k-1)(k-2)}$. The critical point here is that there are only $k-1$ variables instead of the usual $k$. From here we may proceed as in the concluding paragraph of the proof of \cite[Lemma 3.1]{Woo1992} to obtain the upper bound
$$T_3(\xi)\ll P^{\frac{1}{2}(k-1)(k-2)\tet}L_{s,0}(P,P;\tet;\bfPhi),$$
wherein $\Phi_i(z)=(z-\xi)^i$ $(1\le i\le k)$ is a system of type $(0,P,1)$. On substituting this estimate into (\ref{3.9a}), and thence into (\ref{3.8}), we deduce that
\begin{equation}\label{3.11}
J_{s+k-1,k}(P)\ll T_1+P^{(2s+\ome(k,1)-1)\tet}L_{s,0}(P,P;\tet;\bfPhi).
\end{equation}

\par It remains at this stage to bound $T_1$. But an immediate modification of the argument of case (i) of the proof of \cite[Lemma 3.1]{Woo1992} yields the bound
$$T_1\ll (J_{s+k-1,k}(P))^{1-2/(k-1)}(P^{k-1}J_{s,k}(P))^{2/(k-1)}.$$
The proof of the lemma is thus completed by reference to (\ref{3.11}).\end{proof}

At this point, we define the {\it efficient difference} operator $\Del_i^*$ by
$$\Del_i^*(f(x);h;m)=m^{-i}\left( f(x+hm^k)-f(x)\right).$$
When $0\le d<k-1$, it is useful also to define the exponent
$$\nu(d)=\frac{k-d^*}{2(k-d-1)}.$$
We require one last lemma before moving on to prove Theorem \ref{theorem3.1}. 

\begin{lemma}\label{lemma3.4} Suppose that $1<P^\tet \le Q\le P$, and that the system $(\bfPhi)$ is of type $(d,P,A)$. Write $H=P^{1-k\tet }$. Then there exists a system $(\bfXi)$ of type $(d+1,P,k2^kA)$ with the property that
\begin{align*}
L_{s,d}(P,Q;\tet ;\bfPhi )\ll_A &\, P^{k-d^*}J_{s,k}(QP^{-\tet })\\
&+H^{k-d^*}\left( K_{s,d+1}(P,QP^{-\tet };\bfXi )\right)^{\nu(d)}\left( J_{s,k}(QP^{-\tet })\right)^{1-\nu(d)}.
\end{align*}
\end{lemma}

\begin{proof} We initially follow the argument of the proof of \cite[Lemma 4.1]{Woo1992}, with the modified definitions of $K_s$ and $L_s=L_{s,d}(P,Q;\tet;\bfPhi)$ again entailing only slight and superficial alterations. Thus we deduce that $L_s\ll_A U_0+U_1$, where
\begin{equation}\label{3.12}
U_0=P(L_s)^{1-1/(k-d^*)}\left( J_{s,k}(QP^{-\tet })\right)^{1/(k-d^*)},
\end{equation}
and
\begin{equation}\label{3.13}
U_1=\sum_\bfeta \int_{\dbT^k}\left( \prod_{j=1}^{k-d^*}\sum_{1\le h\le H}W(\eta_j\bfalp ;h)\right) |f(\bfalp ;QP^{-\tet }|^{2s}\d\bfalp .
\end{equation}
Here, the outer summation is over $\bfeta \in \{1,-1\}^{k-d^*}$, and
$$W(\bfalp ;h)=\sum_{1\le z\le P}e\left( \alp_1\Xi_1(z;h;p)+\dots +\alp_k\Xi_k(z;h;p)\right) ,$$
in which $\Xi_i(z;h;p)=\Del_i^*(\Phi_i(z);h;p)$ $(1\le i\le k)$.\par

If $U_0\ge U_1$, then $L_s\ll U_0$, and hence we deduce from (\ref{3.12}) that
$$L_s\ll_A P^{k-d^*}J_{s,k}(QP^{-\tet }).$$
This establishes the conclusion of the lemma unless $U_1>U_0$, in which case $L_s\ll U_1$. But in this situation, an application of H\"older's inequality leads from (\ref{3.13}) to the upper bound $L_s\ll \calV_1^{\nu(d)}\calV_2^{1-\nu(d)}$, where
$$\calV_1=H^{2(k-d-1)}\max_{1\le h\le H}\int_{\dbT^k}\left| W(\bfalp ;h)^{2(k-d-1)}f(\bfalp ;QP^{-\tet })^{2s}\right| \d\bfalp $$
and
$$\calV_2=\int_{\dbT^k}\left| f(\bfalp ;QP^{-\tet })\right|^{2s}\d\bfalp .$$
The desired conclusion in this second case follows upon considering the underlying diophantine equations.
\end{proof}

Although we are now prepared to prove Theorem \ref{theorem3.1}, we take a respite to make some comments concerning the variables 
occurring in its statement. Notice first that for each $j$, $s$ and $J$, we have $\phi (j,s,J)\le 1/k$. One therefore has $\tet_s\le 1/k$, and hence by a simple induction one obtains $\Del_s\le \max \{ 0,\Del_{s-k}\} \le {\textstyle{\frac{1}{2}}}k(k+1)$. The formula (\ref{3.1}) therefore yields positive values for the real numbers $\phi^*$ and $\phi$, and hence $\tet_s>0$. It follows also that $\lam_s\le 2s$.\par

We prove Theorem \ref{theorem3.1} by induction on $s$, the case $s=t$ being assumed. We presently suppose that the conclusion of the theorem holds with $s=t+m(k-1)$ for each integer $m$ with $0\le m\le l$, and then fix $s=t+l(k-1)$. For ease of exposition, we write $\lam $ for $\lam_s$, $\tet=\tet_{s+k-1}$, and $\phi(j,J)=\phi(j,s+k-1,J)$, both with and without decoration by an asterisk. Let $j$ be the least integer with $1\le j\le k$ for which $\tet=\phi (j,1)$. For $J=1,\dots ,j$ define $\phi_J=\phi (j,J)$ as in the statement of Theorem \ref{theorem3.1}. Then, if $\phi_J=1/k$ for some $J<j$, we have $\phi (j,J)=\phi (J,J)$, and one finds successively that $\phi (j,r)=\phi (J,r)$ for $r=J,J-1,\dots ,1$, contradicting the minimality of $j$. Thus $\phi_J<1/k$ for $J<j$. We adopt the notation of writing
$$M_i=P^{\phi_i},\quad H_i=PM_i^{-k},\quad Q_i=P(M_1\dots M_i)^{-1}\quad (1\le i\le j),$$
and additionally adopt the convention that $Q_0=P$. We also take $A_J$ to be a series of sufficiently large (but fixed) real numbers with each ratio $A_J/A_{J-1}$ also sufficiently large.\par

We first prove, inductively, that for $J=j-1,j-2,\dots ,0$, all systems $(\bfPhi )$ of type $(J,P,A_J)$ satisfy the relation
\begin{equation}\label{3.16}
L_{s,J}(P,Q_J;\phi_{J+1};\bfPhi )\ll P^{k-J^*}Q_{J+1}^\lam .
\end{equation}
Observe first that if $(\bfPsi )$ is of type $(j,P,A)$, then a trivial estimate yields
$$K_{s,j}(P,Q_j;\bfPsi )\ll P^{2(k-j^*)}J_{s,k}(Q_j).$$
But for all systems $(\bfPhi )$ of type $(j-1,P,A_{j-1})$, it follows from Lemma \ref{lemma3.4} that
$$L_{s,j-1}(P,Q_{j-1};\phi_j;\bfPhi )\ll P^{k-(j-1)^*}J_{s,k}(Q_j)+P^{k-(j-1)^*}H_j^{k-(j-1)^*}J_{s,k}(Q_j).$$
Consequently, on noting that $\phi (j,j)=1/k$, whence $H_j=1$, we deduce that (\ref{3.16}) follows in the case $J=j-1$.\par

We next assume that (\ref{3.16}) holds for $J\ge 1$, and deduce the corresponding result for $J-1$. We have just established (\ref{3.16}) when $J=j-1$, so we may assume that $J\le j-1$. In these circumstances, Lemma \ref{lemma3.3} shows that all systems $(\bfPsi )$ of type $(J,P,A_J)$ satisfy
$$K_{s,J}(P,Q_J;\bfPsi )\ll P^{k-J^*}J_{s,k}(Q_J)+M_{J+1}^{2s+\ome (k,J^*)-J^*}P^{k-J^*}Q_{J+1}^\lam .$$
Since $\lam_s\le 2s$, we infer from our inductive hypothesis that
$$J_{s,k}(Q_J)\ll Q_J^\lam =(M_{J+1}Q_{J+1})^\lam \le M_{J+1}^{2s}Q_{J+1}^\lam,$$
whence
$$K_{s,J}(P,Q_J;\bfPsi )\ll P^{k-J^*}M_{J+1}^{2s}Q_{J+1}^\lam +M_{J+1}^{2s+\ome (k,J^*)-J^*}P^{k-J^*}Q_{J+1}^\lam .$$
Consequently, for all systems $(\bfPhi )$ of type $(J-1,P,A_{J-1})$, it follows from Lemma \ref{lemma3.4} that
\begin{equation}\label{3.17}
L_{s,J-1}(P,Q_{J-1};\phi_J;\bfPhi )\ll T_3+T_4,
\end{equation}
where $T_3=P^{k-(J-1)^*}Q_J^\lam $, and
\begin{equation}\label{3.19}
T_4=H_J^{k-(J-1)^*}(P^{k-J^*}M_{J+1}^{2s+\ome (k,J^*)-J^*}Q_{J+1}^\lam )^{\nu(J-1)}(Q_J^\lam )^{1-\nu(J-1)}.
\end{equation}
We have assumed that $\phi_J<1/k$ for $J<j$, and hence that $\phi_J=\phi_J^*(j,J)$. From (\ref{3.1}) we therefore find that
\begin{align*}
(2s+\ome (k,J^*)-J^*-\lam)\phi_{J+1}&=(k^2-J^*k+{\textstyle{\frac{1}{2}}}J^*(J^*-1)-\Del_s)\phi (j,J+1)\\
&=2k(k-J^*)\phi_J-(k-J^*).
\end{align*}
We thus deduce from (\ref{3.19}) that $T_4=P^{k-(J-1)^*}Q_J^\lam $, whence (\ref{3.17}) yields
$$L_{s,J-1}(P,Q_{J-1};\phi_J;\bfPhi )\ll P^{k-(J-1)^*}Q_J^\lam .$$
It follows that (\ref{3.16}) holds with $J-1$ replacing $J$, and our secondary inductive hypothesis holds for $J=0,1,\dots,j-1$.\par

We have shown that all systems $(\bfPhi )$ of type $(0,P,A_0)$ satisfy
$$L_{s,0}(P,Q_0;\phi_1;\bfPhi )\ll P^{k-1}Q_1^\lam ,$$
so that by Lemma \ref{lemma3.a}, for all systems $(\bfPsi )$ of type $(0,P,1)$, one has
$$J_{s+k-1,k}(P)\ll P^{k-1+\lam }+M_1^{2s+\ome(k,1)-1}P^{k-1}(P/M_1)^\lam .$$
Then $J_{s+k-1,k}(P)\ll P^{k-1+\lam }+P^{\lam'}$, where
\begin{align*}
\lam'&=\lam (1-\tet)+k-1+(2s+{\textstyle{\frac{1}{2}}}k(k-1)-k)\tet\\
&=2(s+k-1)-{\textstyle{\frac{1}{2}}}k(k+1)+\Del_{s+k-1}.
\end{align*}
Thus we may conclude that the primary inductive hypothesis holds with $s+k-1$ in place of $s$, and so the proof of the theorem is complete.

\section{The computations underlying Corollaries \ref{corollary1.2} and \ref{corollary1.3}} Our first task in completing the computations required to establish Corollaries \ref{corollary1.2} and \ref{corollary1.3} is to compute, for each natural number $k$ with $8\le k\le 20$, permissible exponents $\Del_{s,k}$ for $1\le s\le s^*(k)$, for a suitably chosen integer $s^*(k)$. It transpires that one may take $s^*(k)=6k^2$ for $k$ in the aforementioned interval. Next, we observe that the estimate $J_{k,k+1}(P)\ll P^{k+1+\eps}$, available from \cite[Lemma 5.4]{Hua1965} (and in a much sharper asymptotic form in \cite{VW1997}), implies via H\"older's inequality that the exponent $\Del_{s,k}=\frac{1}{2}k(k+1)-s$ is permissible for $1\le s\le k+1$. We initialise our array of permissible exponents $\Del_{s,k}$ by employing a trivial estimate to deduce that for $k+2\le s\le s^*(k)$, the exponent $\Del_{s,k}=\frac{1}{2}k(k+1)-(k+1)$ is permissible. Our strategy at this point is to employ Theorem \ref{theorem3.1} to compute new permissible exponents $\Del_{s+k-1,k}^*$ from the exponents $\Del_{s,k}$, beginning with the integers $s$ in the interval $1\le s\le k+1$, and then proceeding inductively. For each integer $s$, we take $\Del_{s,k}$ to be the smaller of our previous estimate for this quantity, and the newly computed value $\Del_{s,k}^*$.\par

We add two extra devices to the approach outlined in the first paragraph. First, by employing H\"older's inequality, one may verify that for $1\le t\le k-1$, the exponent
$$\Del_{s+t,k}^{(1)}=\frac{(k-1-t)\Del_{s,k}+t\Del_{s+k-1,k}}{k-1}$$
is permissible. If $\Del_{s+t,k}^{(1)}$ is smaller than our previously stored estimate for $\Del_{s+t,k}$, then we may replace the latter by the former. We therefore introduce this linear interpolation step after computing each $\Del_{s+k-1,k}$. Finally, we make use of the estimate from the second author's work on quasi-diagonal behaviour \cite{Woo1994}. Thus, when $3\le t\le k$, one may obtain a permissible exponent $\Del_{s+t,k}$ as follows. We put $l=[k/2]$ and consider integers $r$ and $t$ with 
$\max\{1,k-r\}\le t<2l$. We then define $u=[s(1-t/(2l))^{-1}+1]$, and put
$$\del_w=w-{\textstyle{\frac{1}{2}}}k(k+1)+\Del_{w,k}\quad (w=s,u).$$
Finally, on putting
$$\tet^*=\frac{2(s\del_u-u\del_s)}{urt+2(s\del_u-u\del_s)}$$
and then $\tet=\max \{\tet^*,1/r\}$, we find from \cite[equation (4.8)]{Woo1994} that the exponent
$$\Del_{s+t,k}=\del_s(1-\tet)+(s+{\textstyle{\frac{1}{2}}}(r+t-k-1)(r+t-k))\tet+{\textstyle{\frac{1}{2}}}k(k+1)-(s+t)$$
is permissible. Should any of the exponents obtained through application of these methods be smaller than our previously stored estimates, then we replace the latter by the former. Finally, having computed new estimates for $\Del_{s,k}$ for $k+2\le s\le s^*(k)$, we repeat the computation all over again until we achieve numerical convergence.\par

Next, having computed arrays of permissible exponents $\Del_{s,k}$ for $8\le k\le 20$ and $1\le s\le s^*(k)$, we apply Theorem \ref{theorem1.1} to compute the exponent $\sig(k)$. Note that the computation of $\sig(k)$ makes use of permissible exponents $\Del_{s,k-1}$ corresponding to degree $k-1$. These calculations are reported in Corollary \ref{corollary1.2}. Finally, in order to calculate upper bounds for $\Gtil(k)$, we make use of \cite[Lemma 5.4]{For1995}, so that
$$\Gtil(k)\le \min_{1\le m\le k}\min_{1\le s\le s^*(k)}\lceil 2s+m(m-1)+\Del_{s,k}/(m\sig (k))\rceil .$$
This calculation involves minimising an expression over the $k$ available choices for $m$ as well as the variable $s$. The outcome of these calculations is reported in Corollary \ref{corollary1.3}.

\bibliographystyle{amsbracket}
\providecommand{\bysame}{\leavevmode\hbox to3em{\hrulefill}\thinspace}

\end{document}